\documentclass{amsart}

\usepackage[left=3cm,top=3cm,right=3cm,bottom=3cm]{geometry}
\usepackage{amssymb, amsmath, amsthm, units}
\usepackage{verbatim}
\usepackage{comment}
\usepackage{graphicx}
\usepackage{enumerate}
\usepackage{color}

\DeclareGraphicsExtensions{.pdf,.jpg,.png}

\theoremstyle{plain}
\newtheorem{theorem}{Theorem}
\newtheorem{lemma}[theorem]{Lemma}

\newtheorem{claim}[theorem]{Claim}

\theoremstyle{definition}

\newtheorem{remark}[theorem]{Remark}

\newcommand{\cond}{\textnormal{cond}}
\newcommand{\eps}{\varepsilon}

\newcommand{\R}{\mathbb{R}}
\newcommand{\Cplx}{\mathbb{C}}
\newcommand{\A}{\mathcal{A}}
\newcommand{\Ainv}{\mathcal{A}^{-1}}
\newcommand{\B}{\mathcal{B}}
\newcommand{\C}{\mathcal{C}}

\newcommand{\I}
{\mathcal{I}}
\newcommand{\K}
{\mathcal{K}}
\newcommand{\RR}{\mathcal{R}}
\newcommand{\SSS}{\mathcal{S}}

\newcommand{\set}[1]{\left\{#1\right\}}
\newcommand{\size}[1]{\left|#1\right|}
\newcommand{\abs}[1]{\left|#1\right|}
\newcommand{\norm}[1]{\left\|#1\right\|}

\makeatletter

\begin{document}

\title{An application of kissing number in sum-product estimates}

\author{J\'{o}zsef Solymosi}
\address{\noindent Department of Mathematics, University of British Columbia, 1984 Mathematics Road, Vancouver, BC, V6T 1Z2, Canada}
\email{solymosi@math.ubc.ca}
\thanks{The first author was supported by NSERC, ERC-AdG. 321104, and OTKA NK 104183 grants}

\author{Ching Wong}
\address{\noindent Department of Mathematics, University of British Columbia, 1984 Mathematics Road, Vancouver, BC, V6T 1Z2, Canada} 
\email{ching@math.ubc.ca}

\date{}

\begin{abstract}
The boundedness of the kissing numbers of convex bodies has been known to Hadwiger \cite{HAD} for long. We present an application of it to the sum-product estimate
$$\max(\size{\A+\A},\size{\A\A})\gg\dfrac{\size{\A}^{4/3}}{\lceil\log\size{\A}\rceil^{1/3}}$$
 for finite sets $\A$ of quaternions and of a certain family of well-conditioned matrices.
\end{abstract}

\maketitle

\section{Introduction}

\subsection{Kissing number}
The $n$-th kissing number $c_n$ is the maximum possible number of non-overlapping unit spheres simultaneously touching another given unit sphere in $\R^n$. It is clear that $c_1=2$ and $c_2=6$. The first complete proof of $c_3=12$ was due to Sch\"{u}tte and van der Waerden \cite{SCHVAN}. Musin found the exact value of $c_4$, which is the current largest known kissing number.

\begin{theorem}[Musin \cite{MUS}]\label{Musin}
$c_4=24$.
\end{theorem}

More generally, Hadwiger considered a notion of kissing number of general convex bodies. For a convex body $\K$, we denote by $H(\K)$ the maximum possible number of non-overlapping translates of $\K$ all touching $\K$. See \cite{ZON} for more variations of kissing numbers of convex bodies.

\begin{theorem}[Hadwiger \cite{HAD}]\label{Hadwiger}
If $\K$ is a $d$-dimensional convex body, then $H(\K)\leq3^d-1$.

In particular, if $\K$ is a closed ball in any topology induced by a norm defined on $\Cplx^{d+1}$, then $H(\K)\leq3^d-1$.
\end{theorem}

We apply these two theorems to prove a sum-product estimate, generalizing previous work.

\subsection{Sum-product problems}
Given a finite set $\A$ of a ring, the \emph{sumset} and the \emph{productset} are defined by
$$\A+\A=\set{A+B:A,B\in\A},$$
and
$$\A\A=\set{AB:A,B\in\A}.$$

It was conjectured by Erd\H{o}s and Szemer\'{e}di \cite{ERDSZE} that every finite set $\A$ of integers having large enough cardinality, there holds
\begin{equation}\label{conjecture}
\max(\size{\A+\A},\size{\A\A})\geq\size{\A}^{2-\eps},
\end{equation}
where $\eps\to0$. They proved that
\begin{equation}\label{sumproduct}
\max(\size{\A+\A},\size{\A\A})\gg\size{\A}^{1+\delta},
\end{equation}
for some $\delta>0$. In this paper, the notation $\gg$ is used when there is a hidden constant which does not depend on $\size{\A}$. As usual, the notation $\gg_{a}$ emphasizes that the hidden constant does depend on the parameter $a$.

In recent years, considerable research has been devoted to finding lower bounds on $\delta$ in (\ref{sumproduct}) of other rings and fields. The bound $\delta\geq1/54$ was found by Chang \cite{CHA1} for the ring of quaternions. For complex numbers, the best known bound to date is $\delta>1/(3+\eps)$, proved by Konyagin and Rudnev \cite{KONRUD}. In the case of real numbers, Konyagin and Shkredov \cite{KONSHK} gave an even better bound $\delta\geq1/3+c$, where $c$ is an absolute constant. For more details on the sum-product problem, we refer the interested readers to a recent survey \cite{GRASOL}.

In this paper, we prove that $\delta$ can be arbitrarily close to $1/3$ for quaternions. 

\begin{theorem}\label{thmquaternion}
Let $\A$ be a finite set of quaternions. Then,
\begin{equation}\label{estimate}
\max(\size{\A+\A},\size{\A\A})\gg\dfrac{\size{\A}^{4/3}}{\lceil\log\size{\A}\rceil^{1/3}}.
\end{equation}
\end{theorem}

This sum-product estimate (\ref{estimate}) in the reals was first achieved by second author in \cite{SOL}. He used an observation that $(a+c)/(b+d)$ lies between $a/b$ and $c/d$ for positive real numbers. Regarding $(a+c)/(b+d)$ as an element in $(\A+\A)\times(\A+\A)$, Solymosi showed that
$$\size{(\A+\A)\times(\A+\A)}\gg\dfrac{\size{\A}^4}{\size{\A\A}\log\lceil\size{\A}\rceil}.$$
Konyagin and Rudnev \cite{KONRUD} then generalized this result to the complex plane, by showing that the complex number $(a+b)/(c+d)$ lies in an open set and that some pairwise disjoint open sets can be carefully chosen to contain many elements in $(\A+\A)\times(\A+\A)$. With this result they improved an earlier bound  by the second author \cite{SOL2}. We combine the  ideas from \cite{SOL,SOL2} to prove Theorem \ref{thmquaternion}. A generalization of the previous ideas to the ring of quaternions is that $(a+c)(b+d)^{-1}$ lies in the closed ball centred at $ab^{-1}$ with radius $\norm{ab^{-1}-cd^{-1}}$, as long as $b$ and $d$ are in the same hexadecant. Some of these balls, each containing many elements of the form $(a+c)(b+d)^{-1}$, are chosen. Unlike the case for real numbers, these closed balls may intersect, and so it may happen that some elements are counted many times. We use the tool of kissing number to show that these balls cannot intersect too much, namely, by a corollary (Lemma \ref{lemmakissing}) of Theorem \ref{Musin}, each element can be counted by at most 25 times. 

The same proof, supplemented by Theorem \ref{Hadwiger}, allows us to obtain the same bound for a certain family of well-conditioned matrices. The sum-product estimate for matrices was first studied by Chang \cite{CHA2}, who noted that the sum-product conjecture (\ref{conjecture}) does not hold for square matrices having determinant 1, and proved that, assuming
\begin{equation}\label{detA-B}
\det(A-B)\neq0,\quad\mbox{for any distinct $A,B\in\A$,}
\end{equation}
$\max(\size{\A+\A},\size{\A\A})$ is larger than $f(\size{\A})\size{\A}$, where $f(\size{\A})$ goes to infinity with $\size{\A}$. Tao \cite{TAO} showed that the function $f$ grows polynomially. Solymosi and Vu \cite{SOLVU} then proved that $\delta\geq1/4$ for the ring of $k$ by $k$ matrices satisfying (\ref{detA-B}), where the hidden constant depends on $k$ as well as the largest condition number of matrices in $\A$. A result of Solymosi and Tao in \cite{SOLTAO} implies that $\delta>1/(4+\eps)$ by assuming only (\ref{detA-B}). Under different assumptions, we prove that $\delta>1/(3+\eps)$, as stated below.

\begin{theorem}\label{thmmatrix}
Let $\A$ be a finite set of $k$ by $k$ invertible matrices with complex entries. Suppose that if $A,B,C,D\in\A$, then either  $AB^{-1}=CD^{-1}$ or the block matrix
$\begin{pmatrix}
A & C\\
B & D
\end{pmatrix}$
is invertible. Then,
$$\max(\size{\A+\A},\size{\A\A})\gg_{k,M}\dfrac{\size{\A}^{4/3}}{\lceil\log\size{\A}\rceil^{1/3}},$$
where $M=\max\limits_{A\in\A}\cond(A)=\max\limits_{A\in\A}\norm{A}\norm{A^{-1}}$ is the largest condition number of matrices in $\A$.
\end{theorem}

Unlike quaternions, there is no multiplicative norm on the set of matrices. For easier calculations, we use the operator $1$-norm on our matrices, i.e. $\norm{A}=\max\limits_{1\leq{j}\leq{k}}\sum\limits_{i=1}^k\abs{A_{ij}}$.

\begin{remark}
We cannot omit the assumption that either $AB^{-1}=CD^{-1}$ or the block matrix
$\begin{pmatrix}
A & C\\
B & D
\end{pmatrix}$
is invertible. Indeed, the following matrix families, inspired by Chang \cite{CHA2}, give a small sumset and productset. Take
$$\A_n=\left\{\begin{pmatrix}
1 & i/n\\
0 & 1
\end{pmatrix}:1\leq{i}\leq{n}\right\},$$
then it is easy to see that $\size{\A_n+\A_n}=\size{\A_n\A_n}=2\size{\A_n}-1$, and $\cond\left(\begin{pmatrix}
1 & i/n\\
0 & 1
\end{pmatrix}\right)=(1+i/n)^2\leq4$.
\end{remark}

Since quaternions can be represented by 4 by 4 real matrices, and since all quaternions have condition number 1 (using the quaternion norm), Theorem \ref{thmquaternion} is implied by Theorem \ref{thmmatrix}. We note that the proofs of the two theorems are based on the same idea. We will prove Theorem \ref{thmquaternion}, as a toy version of Theorem \ref{thmmatrix}, in the next section. Then, we describe the necessary modifications for general matrices in Section \ref{sectionmatrix}.

\section{Quaternions}\label{sectionquaternion}
In this section, we are going to prove \begin{equation}\label{quaternionstatement}
\size{\A+\A}^2\size{\A\A}\gg\dfrac{\size{\A}^4}{\lceil\log\size{\A}\rceil},
\end{equation}
where $\A$ is a finite set of quaternions. This immediately implies Theorem \ref{thmquaternion}.

By adjusting the constant hidden in (\ref{quaternionstatement}), it suffices to prove the inequality for a positive fraction of $\A$. This allows us to assume that $0\not\in\A$ and that the quaternions in $\A$ are in the same hexadecant, defined below. The former assumption implies that all elements in $\A$ are invertible and the latter one ensures that the sum of any two elements in $\A$ is invertible and is used in proving Lemma \ref{lemmanormquaternion}.

The idea of hexadecants is an analogue of quadrants in the complex plane. Every quaternion $a=w+xi+yj+zk$ can be viewed as a vector $(w,x,y,z)$ in $\R^4$. We say that some quaternions are in the \emph{same hexadecant} if in each of the four coordinates in the vector representation, they are all non-negative or all non-positive. Pigeonholing the finite set $\A$, we get a subset $\tilde{\A}$ of $\A$ of size at least $\size{\A}/16$ so that the quaternions in $\tilde{\A}$ are in the same hexadecant.

Instead of working with the productset $\A\A$ directly, we will make use of the \emph{ratioset} $\A/\A$ defined as follows:
$$\A/\A=\A\Ainv\cap{\Ainv\A}.$$

The multiplicative energy of $\A$, namely
\begin{equation*}\begin{split}
E(\A)\
&=\size{\set{(a,b,c,d)\in\A^4:ca=db}}\\
&=\size{\set{(a,b,c,d)\in\A^4:ab^{-1}=c^{-1}d}},
\end{split}\end{equation*}
serves as a link between the productset $\A\A$ and the ratioset $\A/\A$ by using the Cauchy-Schwart inequality:
$$E(\A)\geq\dfrac{\size{\A}^4}{\size{\A\A}}.$$
Therefore, in order to prove (\ref{quaternionstatement}), it suffices to prove
\begin{equation}\label{EA,AA}
\dfrac{E(\A)}{\log\lceil\size{\A}\rceil}\ll\size{\A+\A}^2.
\end{equation}

We first express the multiplicative energy $E(\A)$ by the number of representatives in $\A/\A$. Since the multiplication of quaternions is not commutative, for each $x\in{\A/\A}$, the number of representatives of $x\in\Ainv\A$ and that of $x\in\A\Ainv$ are, in general, not the same. We denote by $\ell(x)$ and $r(x)$ these two numbers, i.e. $$\ell(x)=\size{x\A\cap\A}\quad\text{and}\quad{r(x)}=\size{\A\cap{x\A}}.$$
This allows us to write
$$E(\A)=\sum_{x\in\A/\A}\ell(x)r(x).$$

In the following, we assume that
\begin{equation}\label{leftrightassumption}
\sum_{\substack{x\in{\A/\A}\\\ell(x)\geq{r(x)}}}\ell(x)r(x)\geq\sum_{\substack{x\in{\A/\A}\\\ell(x)\leq{r(x)}}}\ell(x)r(x),
\end{equation}
otherwise our results can be proved similarly, and therefore omited.

With such assumption, the multiplicative energy $E(\A)$ can be estimated as:
\begin{equation}\label{EA,ell}
E(\mathcal{A})=\sum_{x\in{\mathcal{A}/\mathcal{A}}}\ell(x)r(x)\leq2\sum_{\substack{x\in{\mathcal{A}/\mathcal{A}}\\\ell(x)\geq{r(x)}}}\ell(x)r(x)\leq2\sum_{\substack{x\in{\mathcal{A}/\mathcal{A}}\\\ell(x)\geq{r(x)}}}\ell(x)^2.
\end{equation}

Using the pigeonhole principle and the fact that $1\leq\ell(x)\leq\size{\A}$, there is an index, $I$, so that
those elements in
$\RR:=\set{x\in\A/\A:\ell(x)\geq{r(x)}\,\,\text{and}\,\,2^I\leq\ell(x)<2^{I+1}}$
contribute at least $1/\lceil\log_2\size{\A}\rceil$ of the sum on the right hand side of (\ref{EA,ell}). Mathematically,
\begin{equation}\label{EA,R}
\dfrac{E(\A)}{2\lceil\log|\A|\rceil}\leq\sum_{x\in\RR}\ell(x)^2<|\RR|2^{2I+2}.
\end{equation}

From now on, we restrict our attention to $\RR$. For each $x\in\RR$, we will form a set $\SSS_x$ consisting of elements in $(\A+\A)\times(\A+\A)$. We will then show that the union $\bigcup_{s\in\RR}\SSS_x$, as a subset of $(\A+\A)\times(\A+\A)$, has size $\gg{E(\A)/\lceil\log|\A|\rceil}$.

For each $x\in\RR$, we define $\SSS_x$ by
$$\SSS_x=\set{(a+c,b+d)\in(\A+\A)\times(\A+\A):ab^{-1}=x\,\,\text{and}\,\,cd^{-1}=\phi(x)},$$
where $\phi:\RR\to\RR$ is a map sending an element to a closest element in $\RR$, i.e.
$$\norm{x-\phi(x)}\leq\norm{x-y},$$
for all $x,y\in\RR$.

The following lemma guarantees that
\begin{equation}\label{sizeofSx}
\size{\SSS_x}\geq\ell(x)\ell(\phi(x))\geq2^{2I}.
\end{equation}

\begin{lemma}\label{lemmauniquenessquaternion}
Given $p,q\in{\A+\A}$ and distinct $x,y\in{\A\Ainv}$. There is at most one quadruple $(a,b,c,d)$ in ${\A^4}$ satisfying
$$a+c=p,\quad{b+d=q},\quad{ab^{-1}=x},\quad{cd^{-1}=y}.$$
\end{lemma}

\begin{proof}
Since $y-x\neq0$, the inverse of $y-x$ exists in the ring of quaternions. One can recover the quaternions $a,b,c,d$ uniquely by direct solving:
$$d=(y-x)^{-1}(p-xq),\quad{b=q-d},\quad{c=yd},\quad{a=xb}.$$
\end{proof}

The two lemmata below together imply that no element in $(\A+\A)\times(\A+\A)$ can appear in too many $\SSS_x$'s. Indeed, Lemma \ref{lemmanormquaternion} shows that if $(a+c,b+d)\in\SSS_x$, then the quaternion $(a+c)(b+d)^{-1}$ is contained in $\B_x$, where $\B_x$ is the closed ball centred at $x$ with radius $\norm{\phi(x)-x}$. Consider the collection of balls $\B_x$, where $x\in\RR$. Notice that none of these balls contains another centre of these balls in the interior. It is then shown in Lemma \ref{lemmakissing} that no quaternion, viewed as a vector in $\R^4$, can be contained simultaneously in too many such closed balls, establishing our claim. We note here that the quaternion norm is essentially the Euclidean norm in $\R^4$.

\begin{lemma}\label{lemmanormquaternion}
Let $a,b,c,d$ be quaternions. Suppose that $b$ and $d$ are non-zero quaternions in the same hexadecant. Then,
$$\norm{(a+c)(b+d)^{-1}-ab^{-1}}\leq\norm{cd^{-1}-ab^{-1}}.$$
\end{lemma}

\begin{proof}
Since $b(b+d)^{-1}+d(b+d)^{-1}=1$, we can write $(b+d)^{-1}=b^{-1}\left(1-d(b+d)^{-1}\right)$. Hence,
\begin{equation*}\begin{split}
\norm{(a+c)(b+d)^{-1}-ab^{-1}}
&=\norm{ab^{-1}\left(1-d(b+d)^{-1}\right)+c(b+d)^{-1}-ab^{-1}}\\
&=\norm{(cd^{-1}-ab^{-1})d(b+d)^{-1}}\\
&=\norm{cd^{-1}-ab^{-1}}\dfrac{\norm{d}}{\norm{b+d}}\\
&\leq\norm{cd^{-1}-ab^{-1}},
\end{split}\end{equation*}
because the norm of quaternions is multiplicative and the inequality $\norm{b+d}\geq\norm{d}$ is a consequence of the assumption that $b$ and $d$ are in the same hexadecant.
\end{proof}

The next lemma is a probably well-known corollary Theorem \ref{Musin}. For the sake of completeness, we include a proof of this lemma.

\begin{lemma}\label{lemmakissing}
Let $P$ be a point in $\R^4$ which is contained in closed balls $\B_1,\ldots,\B_m$ simultaneously, where $\B_i$ has centre $Q_i\in\R^4$. Suppose that $Q_i$ is not contained in the interior of $\B_j$, unless $i=j$. Then $m\leq25$.
\end{lemma}

\begin{proof}
It suffices to construct, in $\R^4$, $m-1$ non-overlapping spheres $S_i$, each of which touches another sphere $S$, and all these $m$ spheres have the same radius.

Reorder the spheres if necessary, we may assume that $P\neq{Q_i}$ for $1\leq{i}\leq{m-1}$ and that
$$0<r:=\norm{Q_1-P}\leq\norm{Q_i-P},\quad\mbox{for all $1\leq{i}\leq{m-1}$.}$$

Now, for each $1\leq{i}\leq{m-1}$, let $S_i$ be the sphere centred at
$$C_i=P+\dfrac{r}{\norm{Q_i-P}}(Q_i-P)$$
with radius $r/2$. We note that every such sphere $S_i$ touches the sphere $S$ centred at $P$ with radius $r/2$.

It remains to show that the spheres $S_1,\ldots,S_{m-1}$ are non-overlapping. Suppose $S_i$ and $S_j$ share some interior point, i.e. $\norm{C_i-C_j}<{r}$. Without loss of generality, assume that
$$\ell_i:=\norm{Q_i-P}\leq\norm{Q_j-P}=:\ell_j.$$
Then,
\begin{equation*}\begin{split}
\norm{Q_i-Q_j}
&=\norm{\left(P+\dfrac{\ell_i}{r}(C_i-P)\right)-\left(P+\dfrac{\ell_j}{r}(C_j-P)\right)}\\
&=\dfrac{1}{r}\norm{\ell_i(C_i-C_j)+(\ell_i-\ell_j)(C_j-P)}\\
&\leq\dfrac{1}{r}\left(\ell_i\norm{C_i-C_j}+(\ell_j-\ell_i)\norm{C_j-P}\right)\\
&<\dfrac{1}{r}\left(\ell_ir+(\ell_j-\ell_i)r\right)\\
&=\ell_j=\norm{Q_j-P}\leq\mbox{the radius of $\B_j$},
\end{split}\end{equation*}
which implies that $Q_i$ is contained in the interior of the ball $\B_j$, and so $i=j$, as desired.
\end{proof}

By Lemma \ref{lemmanormquaternion}, Lemma \ref{lemmakissing}, (\ref{sizeofSx}) and (\ref{EA,R}), we have
$$\size{\A+\A}^2=\size{(\A+\A)\times(\A+\A)}\geq\size{\bigcup_{x\in\RR}\SSS_x}\geq\dfrac1{25}\sum_{x\in\RR}\size{\SSS_x}\geq\dfrac{1}{25}\size{\RR}2^{2I}\geq\dfrac{E(\A)}{200\lceil\log|\A|\rceil},$$
proving the inequality (\ref{EA,AA}), and hence Theorem \ref{thmquaternion}.

\section{General matrices}\label{sectionmatrix}
Again, we are going to prove
\begin{equation}\label{matrixstatement}
\size{\A+\A}^2\size{\A\A}\gg_{k,M}\dfrac{\size{\A}^4}{\lceil\log\size{\A}\rceil},
\end{equation}
where $\A$ is a finite set of matrices satisfying the assumptions stated in Theorem \ref{thmmatrix}.

By adjusting the constant hidden in (\ref{matrixstatement}), it suffices to prove the inequality for a positive fraction of $\A$, where the fraction depends on $k$ and $M$. In the case of quaternions, we required the quaternions to be in the same hexadecant to prove Lemma \ref{lemmanormquaternion}. Here we assume that the matrices in $\A$ are in the same \emph{class}, defined below after some observations, in order to prove Lemma \ref{lemmanormmatrix}, which is the matrix version of Lemma \ref{lemmanormquaternion}.

Let $\rho:\A\to\Cplx$ be a function sending a matrix $A$ to a $k$-th root of $\det(A)$. We denote by $\tilde{A}$ the normalized matrix $A/\rho(A)$, which has determinant 1 and $\cond(\tilde{A})=\norm{A/\rho(A)}\norm{(A/\rho(A))^{-1}}=\cond(A)\leq{M}$. The following claim shows that all entries (real part and imaginary part) of $\tilde{A}$ lie in a bounded interval $[-Mk,Mk]$.

\begin{claim}\label{claimAtilde}
Let $\tilde{A}$ be a $k$ by $k$ matrix so that $\det(\tilde{A})=1$ and $\cond(\tilde{A})\leq{M}$. Then,
$$1/k\leq\norm{\tilde{A}}\leq{Mk}.$$ In particular, the real part and imaginary part of every entry of $\tilde{A}$ have norm at most $Mk$, i.e.
$$\abs{\Re(\tilde{A}_{ij})}\leq{Mk}\quad\text{and}\quad\abs{\Im(\tilde{A}_{ij})}\leq{Mk},\quad\mbox{for all $i,j$.}$$
\end{claim}

\begin{proof}
Suppose, on the contrary, that $\norm{\tilde{A}}<1/k$. Then, we have $\abs{\tilde{A}_{ij}}<1/k$ for all $i,j$, and so $1=\abs{\det(\tilde{A})}<k!(1/k)^k<1$, contradiction.

Since $\det(\tilde{A}^{-1})=1$ and $\cond(\tilde{A}^{-1})=\cond(\tilde{A})\leq{M}$, we also have $\norm{\tilde{A}^{-1}}\geq1/k$. Therefore,
$$\norm{\tilde{A}}=\dfrac{\cond{\tilde{A}}}{\norm{\tilde{A}^{-1}}}\leq\dfrac{M}{1/k}=Mk,$$
as desired.
\end{proof}

We note that the determinant as a function from all $k$ by $k$ complex matrices to $\Cplx$ is continuous. Thus, it is uniformly continuous from the compact set $\C$ consisting of all $k$ by $k$ complex matrices which have norm at most $2Mk$, i.e. there exists $\delta>0$ (depending only on $M$ and $k$) so that if $A,B\in\C$ with $\norm{A-B}<\delta$, then $\abs{\det(A)-\det(B)}<1/2$.

Let $\eps=\min(\delta/(3Mk^2),1/(6M^2k^3))$ (derived in the proof of Claim \ref{claimB+Dinvertible} and Lemma \ref{lemmanormmatrix}). Using the fact that the interval $[-Mk,Mk]$ can be written as a disjoint union of $\lceil1/\eps\rceil$ intervals $\I_1,\ldots,\I_{\lceil1/\eps\rceil}$ each having length at most $2Mk\eps$, we partition $\A$ into $4\lceil{1/\eps}\rceil^{2k^2}$ classes according to which quadrant the image of the function $\rho$ lies in, as well as the intervals where the $2k^2$ (real part and imaginary part) entries of a normalized matrix lie in. Precisely, the function $\rho$ sends matrices in the same class to the same quadrant, and in each of the $2k^2$ entries, the normalized matrices in the same class all belong to the same interval $\I_j$. It follows by the pigeonhole principle that one of the classes constitutes a positive fraction of $\A$.

Following the proof from Section \ref{sectionquaternion}, we only need to prove the three corresponding lemmata.

The first lemma shows that the size of $\SSS_x$ is large. Here the assumption that, for all $A,B,C,D\in\A$, either  $AB^{-1}=CD^{-1}$ or the block matrix
$\begin{pmatrix}
A & C\\
B & D
\end{pmatrix}$
has non-zero determinant is used. We remark that this assumption is equivalent to the assumption that, for all $A,B,C,D\in\A$, either $A^{-1}B=C^{-1}D$ or the block matrix
$\begin{pmatrix}
A & C\\
B & D
\end{pmatrix}$
has non-zero determinant. Therefore, our proof still works if (\ref{leftrightassumption}) does not hold.

\begin{lemma}\label{lemmauniquenessmatrix}
Given $P,Q\in{\A+\A}$ and distinct $X,Y\in{\A\Ainv}$. There is at most one quadruple $(A,B,C,D)$ in ${\A^4}$ satisfying
$$A+C=P,\quad{B+D=Q},\quad{AB^{-1}=X},\quad{CD^{-1}=Y}.$$
\end{lemma}

\begin{proof}
One may follow the proof of Lemma \ref{lemmauniquenessquaternion}, provided $Y-X$ is invertible. Indeed, if $X=AB^{-1}$ and $Y=CD^{-1}$, then
$$0\neq\det\begin{pmatrix}
A & C\\
B & D
\end{pmatrix}=\det(D)\det(A-CD^{-1}B)=\det(D)\det(AB^{-1}-CD^{-1})\det(B)$$
implies $CD^{-1}-AB^{-1}$ is invertible.
\end{proof}

As we mentioned earlier in this section, the second lemma relies on the assumption that the matrices come from the same class of $\A$. In section \ref{sectionquaternion}, we used the fact that the sum of two non-zero quaternions in the same hexadecant is invertible. The next claim guarantees that $(B+D)^{-1}$ exists, as long as $B$ and $D$ are in the same class.

\begin{claim}\label{claimB+Dinvertible}
If $B,D\in\A$ are in the same class, then $B+D$ is invertible.
\end{claim}

\begin{proof}
Write $b=\rho(B)$ and $d=\rho(D)$. Since $b,d\in\Cplx$ are in the same quadrant, we have $b+d\neq0$ and $\abs{b}<\abs{b+d}$. We write $\tilde{B}=B/b$ and $\tilde{D}=D/d$ to get
\begin{equation}\label{equationB+D}
B+D=b\tilde{B}+d\tilde{D}=b\left(\tilde{B}-\tilde{D}\right)+(b+d)\tilde{D}=(b+d)\left(\dfrac{b}{b+d}\left(\tilde{B}-\tilde{D}\right)+\tilde{D}\right).\end{equation}
According to the way we partition $\A$ into classes, we know that the entries (real part and imaginary part) of $\tilde{B}-\tilde{D}$ are between $-2Mk\eps$ and $2Mk\eps$, and hence
\begin{equation}\begin{split}\label{equation2keps}
\norm{\dfrac{b}{b+d}\left(\tilde{B}-\tilde{D}\right)}
&=\dfrac{\abs{b}}{\abs{b+d}}\max\limits_{1\leq{j}\leq{k}}\sum\limits_{i=1}^k\abs{(\tilde{B}-\tilde{D})_{ij}}<{k}\max\limits_{1\leq{i,j}\leq{k}}\abs{(\tilde{B}-\tilde{D})_{ij}}\\
&\leq{k}\sqrt{(2Mk\eps)^2+(2Mk\eps)^2}<3Mk^2\eps\leq\delta.
\end{split}\end{equation}
By Claim \ref{claimAtilde}, $\norm{\tilde{D}}\leq{Mk}$ and $\norm{\frac{b}{b+d}(\tilde{B}-\tilde{D})+\tilde{D}}\leq\norm{\frac{b}{b+d}(\tilde{B}-\tilde{D})}+\norm{\tilde{D}}<3Mk^2\eps+Mk<2Mk$. Hence, both $\tilde{D}$ and $\frac{b}{b+d}(\tilde{B}-\tilde{D})+\tilde{D}$ are belong to $\C$. Recall that $\det(\tilde{D})=1$. Using the uniform continuity of determinant from $\C$ to $\Cplx$, the norm of the determinant of $\frac{b}{b+d}(\tilde{B}-\tilde{D})+\tilde{D}$ is at least $1/2$. Therefore, $B+D$ is invertible by (\ref{equationB+D}).
\end{proof}

\begin{lemma}\label{lemmanormmatrix}
Let $A,B,C,D$ be matrices in the same class of $\A$. Then,
$$\norm{(A+C)(B+D)^{-1}-AB^{-1}}\leq\norm{CD^{-1}-AB^{-1}}.$$
\end{lemma}

\begin{proof}
From the proof of Lemma \ref{lemmanormquaternion} and the submultiplicativity of the norm chosen, we have
$$\norm{(A+C)(B+D)^{-1}-AB^{-1}}\leq\norm{CD^{-1}-AB^{-1}}\norm{D(B+D)^{-1}}.$$
Write $b=\rho(B)$ and $d=\rho(D)$ as before. Denoting by $I_k$ the $k$ by $k$ identity matrix, we have
\begin{equation*}\begin{split}
\norm{D(B+D)^{-1}}
&=\dfrac{\abs{d}}{\abs{b+d}}\norm{\tilde{D}\left(\tilde{D}+\dfrac{b}{b+d}\left(\tilde{B}-\tilde{D}\right)\right)^{-1}}\\
&=\dfrac{\abs{d}}{\abs{b+d}}\norm{I_k-\dfrac{b}{b+d}\left(\tilde{B}-\tilde{D}\right)\left(I_k+\dfrac{b}{b+d}\tilde{D}^{-1}\left(\tilde{B}-\tilde{D}\right)\right)^{-1}\tilde{D}^{-1}}\\
&\leq\dfrac{\abs{d}}{\abs{b+d}}\left(\norm{I_k}+\dfrac{\abs{b}}{\abs{b+d}}\norm{\tilde{B}-\tilde{D}}\norm{\left(I_k+\dfrac{b}{b+d}\tilde{D}^{-1}\left(\tilde{B}-\tilde{D}\right)\right)^{-1}}\norm{\tilde{D}^{-1}}\right)\\
&\leq\dfrac{\abs{d}}{\abs{b+d}}\left(1+\dfrac{\abs{b}}{\abs{b+d}}(3Mk^2\eps)\norm{\left(I_k+\dfrac{b}{b+d}\tilde{D}^{-1}\left(\tilde{B}-\tilde{D}\right)\right)^{-1}}Mk\right),
\end{split}\end{equation*}
where the first inequality is from the triangle inequality and the submultiplicativity of the norm, whilst the second inequality is from (\ref{equation2keps}) and Claim \ref{claimAtilde}. Recall that $\eps\leq1/(6M^2k^3)$. Rearranging, it suffices to prove
\begin{equation}\label{equationNeumann}
\norm{\left(I_k+\dfrac{b}{b+d}\tilde{D}^{-1}\left(\tilde{B}-\tilde{D}\right)\right)^{-1}}\leq\dfrac{2\abs{b+d}}{\abs{d}}.
\end{equation}
Indeed, using (\ref{equation2keps}), Claim \ref{claimAtilde} and the submultiplicativity of the operator 1-norm again, we have
$$\norm{\dfrac{b}{b+d}\tilde{D}^{-1}\left(\tilde{B}-\tilde{D}\right)}\leq\dfrac{\abs{b}}{\abs{b+d}}Mk(3Mk^2\eps)<\dfrac12<1,$$
which allows us to express the inverse in (\ref{equationNeumann}) by a convergent Neumann series to get
$$\norm{\left(I_k+\dfrac{b}{b+d}\tilde{D}^{-1}\left(\tilde{B}-\tilde{D}\right)\right)^{-1}}\leq\dfrac{1}{1-\norm{\frac{b}{b+d}\tilde{D}^{-1}\left(\tilde{B}-\tilde{D}\right)}}<2\leq\dfrac{2\abs{b+d}}{\abs{d}},$$
proving (\ref{equationNeumann}), and so $\norm{D(B+D)^{-1}}\leq1$.
\end{proof}

We will state but omit the proof of the last lemma needed as it is essentially the same as the proof of Lemma \ref{lemmakissing}, which works in any norm induced topology. Here we regard a $k$ by $k$ complex matrix as a point in $\Cplx^{k^2}$ by arbitrarily fixing an ordering of the $k^2$ entries. We apply Theorem \ref{Hadwiger} to prove this lemma.

\begin{lemma}
Let $P$ be a point in $\Cplx^{k^2}$ which is contained in closed balls $\B_1,\ldots,\B_m$ simultaneously, where $\B_i$ has centre $Q_i\in\Cplx^{k^2}$. Suppose that $Q_i$ is not contained in the interior of $\B_j$, unless $i=j$. Then $m\leq3^{k^2-1}$.
\end{lemma}

By lower bounding the size of $\SSS_x$ as in Section \ref{sectionquaternion}, Theorem \ref{thmmatrix} is proved.

\section{Remarks}
The same counting method works if there are two families of matrices $\mathcal{A}$ and $\mathcal{B}$ and we are interested about the cardinality $|\mathcal{A}+\mathcal{B}|+|\mathcal{A}\mathcal{B}|.$ For the symmetric case $\mathcal{A}=\mathcal{B}$ the weaker condition $\det{(A-B)}\neq 0,$ for any $A\neq B,$ might be sufficient. 
The main question remains open; is there a set of non-singular $k\times k$ matrices, $\mathcal{A},$ with $\det{(A-B)}\neq 0,$ for any $A\neq B,$ such that 
$$\max(\size{\A+\A},\size{\A\A})\leq {\size{\A}^{2-c}}$$
for some $c>0$ ? So far there is no (much) better construction is known for matrices than for integers in the sum-product problem.


\end{document}